\documentclass[12pt,a4paper]{article}   

\usepackage{graphicx}
\usepackage{enumerate}
\usepackage{amssymb, amsmath, amsthm}
\usepackage[round]{natbib}
\usepackage{dsfont}

\usepackage{tikz,pgfplots}
\usetikzlibrary{arrows,automata}

\newtheorem{proposition}{Proposition}[section] 

\newtheorem{assumption}[proposition]{Assumption}
\newtheorem{thm}[proposition]{Theorem}
\newtheorem{lemma}[proposition]{Lemma}

\newtheorem{remark}[proposition]{Remark}

\newtheorem{notation}[proposition]{Notation}

\usepackage[T1]{fontenc}
\usepackage{lmodern}

\newcommand{\N}{\ensuremath{{\mathbb N}}}

\newcommand{\R}{\ensuremath{{\mathbb R}}}

\newcommand{\e}{\ensuremath{{\rm e}}}

\newcommand{\E}{\ensuremath{{\mathbb E}}}
\newcommand{\Pro}{\ensuremath{{\mathbb P}}}

\DeclareMathOperator*{\LambertW}{LambertW}
\newcommand{\ov}[1]{\overline{#1}}
\newcommand{\un}[1]{\underline{#1}}

\begin{document}
\title{Impulse control and expected suprema}
\author{S\"oren Christensen\thanks{Chalmers University of Technology and the University of Gothenburg, Department of Mathematical Sciences,
SE-412 96 Gothenburg, Sweden, email: sorenc@chalmers.se}
, Paavo Salminen\thanks{\AA bo Akademi, Faculty of Science and Engineering, \small FIN-20500 \AA bo, Finland, email: phsalmin@abo.fi}}
\date{\today}
\maketitle

\begin{abstract}
We consider a class of impulse control problems for general underlying strong Markov processes on the real line, which allows for an explicit solution. The optimal impulse times are shown to be of threshold type and the optimal threshold is characterized as a solution of a (typically nonlinear) equation. The main ingredient we use is a representation result for excessive functions in terms of expected suprema. 
\end{abstract}

\vspace{.8cm}


\section{Introduction}

Impulse control problems form an important class of stochastic control problems and find applications in a wide variety of fields ranging from finance, e.g. cash management and portfolio optimization, see \cite{korn1999} and \cite{irlesass2006}, optimal forest management, see \cite{Willassen98}, \cite{A04} and the references therein, and control of an exchange rate by the Central Bank, see \cite{MOk}, \cite{CZ}. The general theory for impulse control problems is often based on the seminal work \cite{BLi}, where the problem is treated using {quasi-variational inequalities}, see also \cite{korn1999} for a survey with focus on financial applications. For a treatment based on the use of superharmonic functions in a general Markovian framework, we refer to \cite{C13impulse}. However, the class of problems that allow for an explicit solution is very limited. Even for underlying one-dimensional diffusion processes, general solutions are only known for subclasses of problems, see \cite{Alvarez04}, \cite{AL}, and \cite{Egami08}. In these references,
 solution methods based on excessive functions were established under assumptions on the reward structure that forces the optimal strategy to be in a certain class. For processes with jumps, the class of explicitly solvable problems is even more scarce. In the monograph \cite{OkSu}, a general verification theorem for jump diffusions and a connection to sequences of optimal stopping problems is established. Some examples are discussed therein for general classes of jump processes and particular reward structures arising in the control of exchange rates and optimal stream of dividends under transaction costs, which allow for a solution using the guess-and-verify-approach. 

On the other hand, the optimal stopping theory for general L\'evy processes was developed in the last decade starting with a treatment of the perpetual American put optimal stopping problem in \cite{Mo}, see also \cite{AK}. Another source was the treatment of the  Novikov-Shiryayev optimal stopping problem 
\[\sup_\tau\E_x\left(e^{-r\tau}\left({X_\tau^+}\right)^m\right),\]
where $X$ is a general L\'evy process and $m$ is a positive constant. The main tools for the treatment are the Wiener-Hopf factorization and, based on this, the use of Appell polynomials associated with the running maximum of $X$ up to an independent exponentially-distributed time, see \cite{NS, KyS, NS2, Sa}. Inspired by these findings, solution techniques were developed in this setting for more general reward functions $g$, see \cite{Su,DLU}. In \cite{CST}, this approach was developed for general underlying Markov processes $X$ on the real line. The starting point is to represent the reward function $g$ in the form
\begin{align}\label{eq:repr_sup}
g(x)=\E_x(f(M_T)),
\end{align}
for some function $f$; here and in the sequal, $M$ denotes the running maximum of $X$ and $T$\ is an independent $Exp(r)$-distributed time. For example, if $g(x)=x^m, m=1,2,...,$ and $X$ a L\'evy process, then $f$ is the $m$-th Appell polynomial associated with $M_T$.

The aim of this paper is to treat in a similar manner the impulse control problem for general underlying Markov processes on the real line and general reward function $g$ (fulfilling certain conditions discussed below) with value function
\begin{align}\label{eq:value_fct}
v(x)=\sup_{S=(\tau_n,\gamma_n)_{n\in\N}}\E_x^S\left(\sum_{n=1}^\infty e^{-r\tau_n}g(X_{\tau_n,-})\right),
\end{align}
where the supremum is taken over all impulse control strategies $S=(\tau_n,\gamma_n)_{n=1}^\infty$ such that the process is restarted at a fixed point $\gamma_n=x_0$ after an impulse is exercised and $X_{\tau_n,-}$ denotes the value before the $n$-th impulse is exercised. A more detailed description is given in the following section. We show that the optimal strategy is of threshold type, where the threshold $x^*$ can be described (semi-)explicitly in terms of certain $c$-values of the function $f$, i.e., $f(x^*)=c$, in the representation \eqref{eq:repr_sup}.

Value functions of the form \eqref{eq:value_fct} arise naturally in many applications. One motivation stems from optimal harvesting problems. Here, the process $X$ is interpreted as, e.g., the random volume of wood in a forest stand. At the intervention points $\tau_n$, the forest is harvested to a base level $x_0$ (we consider w.l.o.g. $x_0=0$ in the following). Selling the wood yields the gross reward $g(X_{\tau_n,-})-g(x_0)$. Here, $g(x_0)\leq 0$ can be interpreted as a fixed cost component -- the costs of one intervention -- so that the net reward is given by $g(X_{\tau_n,-})-g(x_0)+g(x_0)=g(X_{\tau_n,-})$ and the value function above describes the maximal expected discounted reward in this setting. The associated control problem is known as (a stochastic version of) Faustmann's problem. We refer to \cite{A04} for further discussions and references. 

Let us mention that in our analysis below, the case $g(x_0)=0$ is included, i.e. no fixed cost component may be present. This situation is typically not considered in literature on impulse control problems. {The reason is that one expects control strategies with infinite activity on finite time intervals to have higher reward than impulse control strategies if no fixed costs have to be paid for the controls.} Nonetheless, as we will see below, also in this case strategies of impulse control-type turn out to be optimal if the gain function $g$ grows slow enough around $x_0$. 

Following the standard literature on Faustmann's problems, we assume to have a fixed base level $x_0$. In other classes of problems, it seems to be more natural to allow the decision maker to choose this base level also. As we do not want to overload this paper, we do not go into details here. Nonetheless, let us mention that the theory developed here can be used as a building block for the solution of this class of problems also. 

Furthermore, many results in the standard literature on impulse control theory are formulated for rewards including a running cost/reward component of the form
\[\E_x^S\left(\int_0^\infty e^{-rs}k(X_s)ds\right).\]
Note that -- at least in the case that $k$ fulfills a suitable integrability condition -- this is no generalization at all as the problem can be reduced to the form \eqref{eq:value_fct} using the $r$-resolvent operator for the uncontrolled process, see, e.g., \cite{C13impulse}, Lemma 2.1.\\

The structure of the paper is as follows. After summarizing some facts about Hunt processes, an exact description of the problem is given in Subsection \ref{subsec:impulse}. 
The main theoretical findings of this paper are given in Section \ref{sec:sol}. In Subsection \ref{subsec:degen}, we first characterize situations where no optimal strategies exist (in the class of impulse control strategies) and give $\epsilon$-optimal strategies in this case. The non-degenerated case is then treated in Subsection \ref{subsec:non-deg}, where the solution is given under general conditions and Assumption \ref{eq:Mv_v} introduced therein. The validity of this assumption is then discussed for certain classes of processes in Subsection \ref{subsec:cond_6}.
The results are illustrated on different examples for L\'evy processes and reflected Brownian motions in Section \ref{sec:example}.

\section{Preliminaries}

\subsection{Hunt processes}\label{subsec:hunt}
For this paper, we consider a general Markovian framework. More precisely, we assume the underlying process $X=(X_t)_{t\geq
  0}$ to be a Hunt process (with infinite lifetime) on a subset $E$ of the real line ${\R}$. 
	We let $(\mathcal{F}_t)_{t\geq0}$ denote the natural filtration generated by $X$ and $\Pro_x$ and $\E_x$ the probability measure and the expectation, respectively, associated with $X$ when started at $x\in E$. 
In \cite{MoSa}, \cite{Cisse2012}, and \cite{CST}, this class of processes is discussed in the context of optimal stopping.
A more detailed treatment is given in 
 \cite{BG},
 \cite{CW}, and 
  \cite{Sh}. Hunt processes have the following important regularity properties:
they are quasi left continuous strong Markov processes with right
continuous sample paths having left limits. \\

Throughout the paper, the notation
$T$ is used for an exponentially distributed random variable assumed
to be independent of $X$ and having the mean $1/r$, where $r>0$ is the discounting parameter 
for the problem at hand. 
A key result in our approach is the following representation result for excessive functions of $X$, 
 see, e.g., \cite{FK}.

\begin{lemma}\label{lem:sup_exc}
Let $f:E\rightarrow \R_+\cup\{+\infty\}$ be an upper semicontinuous function and define
\[u(x):=\mathbb{E}_x\big(\sup_{0\leq t\leq T}f(X_t)\big).\]
 Then the function
$u: E\rightarrow \mathbb{R}\cup\{+\infty\}$ is $r$-excessive.
\end{lemma}
\begin{proof}
See \cite{CST}, Lemma 2.2.
\end{proof}

The connection between the running maximum of the process $X$ and the first passage times is given by the following result, which corresponds to the well-known fluctuation identities for L\'evy processes as used, e.g., in \cite{NS}, \cite{kyp}, and \cite{KyS}. 
\begin{lemma}\label{lem:M_tau}
Let $f$ and $g$ be real functions such that for all $x$
\begin{align*}
g(x)=\E_x(f(M_T)),
\end{align*}
where $M$ denotes the running maximum of $X$. Furthermore, let $y\in \R$ and
\[\tau_y=\inf\{t\geq 0:X_t\geq y\}.\]
Then, whenever the expectations exist,
\begin{align}\label{eq:harm_threshold}
\E_x\left(f(M_T);M_T\geq y\right)=\E_x(\e^{-r\tau_y}g(X_{\tau_y})).
\end{align}
Here and in the following, we always skip the indicator $1_{\{{\tau_y}<\infty\}}$.
\end{lemma}

\begin{proof}
Note that by the memorylessness property of the exponential distribution
\begin{align*}
\E_x\left(f(M_T);M_T\geq y\right)=&\E_x\left(f(M_T);\tau_y\leq T\right)\\=&\E_x({f}(\sup_{\tau_y\leq t\leq T}X_t);\tau_y\leq T)\\
=&\int_0^\infty \E_x\left(\E_x\left({f}(\sup_{\tau_y\leq t\leq s}X_t)\Big|\mathcal{F}_{\tau_y}\right){1}_{\{\tau_y\leq s\}}\right)\Pro(T\in ds)\\
=&\E_x\left(\E_{X_{\tau_y}}\left(\int_0^\infty{f}(\sup_{0\leq t\leq s}X_t)\Pro(T\in ds)\right)e^{-r\tau_y}\right)\\
=&\E_x(e^{-r\tau_y}\E_{X_{\tau_y}}({f}(M_T)))\\
=&\E_x(\e^{-r\tau_y}g(X_{\tau_y})).
\end{align*}
\end{proof}

\subsection{General impulse control problem}\label{subsec:impulse}

For general Hunt processes on\ $E\subseteq \R$, the exact definition of the right setting for treating impulse control problems including all technicalities is lengthy. For our considerations in this paper, we only explain the objects on an intuitive level, which is sufficient for our further considerations. For an exact treatment, we refer the interested reader to  the appendix in \cite{C13impulse} and the references given there. 

The object is to maximize over {impulse control strategies}, which are sequences $S=(\tau_n,\gamma_n)_{n=1}^\infty$ of times and impulses, respectively. Under the associated family of measures $(\Pro^S_x)_{x\in E}$, the process $X$ is still a strong Markov process, where between each two random times $\tau_{n-1}<\tau_n$, the process runs uncontrolled with the same dynamics as the original Hunt process. At each random time $\tau_n$, an impulse is exercised and the process is restarted at the new random state $X_{\tau_n}=\gamma_n$. It is assumed that $\tau_n$ is a stopping time for the process with only $n-1$ controls and $\gamma_n$ is a random variable measurable with respect to the corresponding pre-$\tau_n$ $\sigma$-algebra. Furthermore, we only consider admissible impulse control strategies such that $\tau_n\nearrow\infty$ as $n\nearrow\infty$.  

Since the underlying process may have jumps, we have to distinguish jumps coming from the dynamics of the process from jumps that arise due to a control, which may take place at the same time. Hence we have $X^n_{\tau_n}\not=X_{\tau_n-}$ in general, where  $X^n$ denotes the process with only $n-1$ controls. For our further developments, it will be important to consider the process $X^n$ also at time point $\tau_n$. Therefore, we write
\[X_{\tau_n,-}:=X^n_{\tau_n}\]
to denote the value of the process at $\tau_n$ if no control would have been exercised. For continuous underlying processes as diffusion processes, we have $X_{\tau_n,-}=X_{\tau_n-}$, which motivates this notation.

To state the class of problems we are interested in and to fix ideas we assume that $0\in E$. The decision maker can restart the process process at $x_0=0$ coming from a positive state and no other actions are allowed. This means that $\gamma_n=0$ and $X_{\tau_n-}>0$ for all\ $n$.
Furthermore, we fix a continuous function $g:[0,\infty)\cap E \rightarrow \R$ which is bounded from below. For each control from $x$ to 0, we get a reward $g(x)$, i.e. we consider the impulse control problem with value function
\begin{align}\label{eq:problem}\tag{ICP}
v(x)=\sup_{S=(\tau_n,\gamma_n)_{n\in\N}}\E_x^S\left(\sum_{n=1}^\infty e^{-r\tau_n}g(X_{\tau_n,-})\right),
\end{align}
and we assume that for all $S$ and $x\in E$
\begin{equation}\label{eq:finite}
\E_x^S\left(\sum_{n=1}^\infty e^{-r\tau_n}g(X_{\tau_n,-})\right)<\infty.
\end{equation}

%

In an even much more general setting, a verification theorem for treating these problems is given in \cite{C13impulse}, Proposition 2.2, which in our situation reads as follows. For the sake of completeness, we also include the proof.

\begin{thm}\label{prop:h_maj} Let $\hat{v}:E\rightarrow \R$ be measurable and define the maximum operator ${\bf M}$ by 
\[{\bf M}\hat{v}(x)=g(x)+\hat{v}(0)\mbox{ for }x\geq0,\;{\bf M}\hat{v}(x)=-\infty\mbox{ for }x<0\]
\begin{enumerate}[(i)]
\item If $\hat{v}$ is nonnegative, $r$-excessive for the underlying uncontrolled process, and ${\bf M}\hat{v}(x)\leq \hat{v}(x)$ for all\ $x$, then it holds that
\[v(x)\leq \hat{v}(x)\mbox{  for all\ $x$}.\] 
\item\label{item:h_maj2} If $x\in E$ and $S=(\tau_n,\gamma_n)_{n=1}^\infty$ is an admissible impulse control strategy such that 
\begin{equation}\label{eq:harmonic}
\E_x^S\left(e^{-r\tau_n}\hat{v}(X_{\tau_n,-})\right)=\E_x^S\left(e^{-r\tau_{n-1}}\right)\hat{v}(0)\mbox{ for all }n\in\N,
\end{equation}
\[\hat{v}(0)+g(X_{\tau_n,-})\geq \hat{v}(X_{\tau_n,-})\;\;\;\Pro_x^S-a.s.,\]
then
\[\hat{v}(x)\leq v(x).\]
\end{enumerate}
\end{thm}

\begin{proof}
Let $S=(\tau_n,\gamma_n)_{n}$ be an arbitrary admissible impulse control strategy.
If $\hat v$ is $r$-excessive, by the optional sampling theorem for nonnegative supermartingales we obtain (keeping in mind that $X$ runs uncontrolled between $\tau_{n-1}$ and $\tau_n$ under $\Pro^S$)
\begin{align*}\E_x^S\left(e^{-r\tau_n}\hat v(X_{\tau_n,-})-e^{-r\tau_{n-1}}\hat v(X_{\tau_{n-1}})\right)\leq 0,\end{align*}
{hence
\begin{align}
&\E_x^S\left(\sum_{n=1}^\infty e^{-r\tau_n}\big(\hat v(X_{\tau_n})-\hat v(X_{\tau_{n},-})\big)\right)+\hat v(x)\nonumber\\
=&-\E_x^S\left(\sum_{n=1}^\infty \big(e^{-r\tau_n}\hat v(X_{\tau_n,-})-e^{-r\tau_{n-1}}\hat v(X_{\tau_{n-1}})\big)\right)\nonumber\\
\geq& 0.\label{eq:harm_appl}
\end{align}}
Using this inequality we get
\begin{align*}
\E_x^S\left(\sum_{n=1}^\infty  e^{-r\tau_n}g(X_{\tau_n,-})\right)
&\leq\E_x^S\left(\sum_{n=1}^\infty  e^{-r\tau_n}\left(g(X_{\tau_n,-})+\hat v(X_{\tau_n})-\hat v(X_{\tau_n,-})\right)\right)+\hat v(x)\\
&=\E_x^S\left(\sum_{n=1}^\infty  e^{-r\tau_n}\left(g(X_{\tau_n,-})+\hat v(0)-\hat v(X_{\tau_n,-})\right)\right)+\hat v(x)\\
&=\E_x^S\left(\sum_{n=1}^\infty  e^{-r\tau_n}\left({\bf M}\hat{v}(X_{\tau_n,-})-\hat v(X_{\tau_n,-})\right)\right)+\hat v(x),
\end{align*}
where $\hat{v}(X_{\tau_n})=\hat{v}(\gamma_n)=\hat{v}(0)$. Since ${\bf M}\hat{v}(y)\leq \hat{v}(y)$ for all\ $y$ we obtain that 
\begin{align*}
\E_x^S\left(\sum_{n=1}^\infty  e^{-r\tau_n}g(X_{\tau_n,-})\right)\leq \hat v(x),
\end{align*}
and consequently  $v(x)\leq \hat v(x)$ because $S$ is arbitrary. (Note that this implies \eqref{eq:finite} also). \\
{Claim (ii) is proved following the steps in the proof of (i): assumption \eqref{eq:harmonic} guarantees equality in \eqref{eq:harm_appl}. }
\end{proof}

In general, it is hard to find a candidate to apply this verification theorem for obtaining an explicit solution. In the following section, we will demonstrate how the representation result for excessive functions presented in Lemma \ref{lem:sup_exc} can be used to succeed.

\section{{Main results}}\label{sec:sol}
For a continuous reward function $g$, we study now \eqref{eq:problem} with value function
\begin{align*}\label{eq:problem}
v(x)=\sup_{S}\E_x^S\left(\sum_{n=1}^\infty e^{-r\tau_n}g(X_{\tau_n,-})\right)
\end{align*}
as described in Subsection \ref{subsec:impulse}.
%
The main ingredient for our treatment of this problem is a representation of $g$ as given in \eqref{eq:repr_sup}. The existence and explicit determination of such a representation is discussed in detail in \cite{CST}, Subsection 2.2. We will come back to this in Section \ref{sec:example} for examples of interest.

\subsection{Degenerated case}\label{subsec:degen}
We first treat the degenerated case, in which we only find $\epsilon$-optimal impulse control strategies {of the form \textit{Wait until the process reaches level $\epsilon$ and then restart the process at state 0}. Taking the limit (in a suitable sense) as $\epsilon\rightarrow 0$, the controlled process is the process reflected at state 0. This limit is obviously not an admissible impulse control strategy, but a strategy  of singular control type.}

Recall that $M$ denotes the running maximum of $X$ and $T$\ is an independent $Exp(r)$-distributed time.

\begin{thm}\label{thm:degenerated}
Assume that $g(0)=0$ 
and that there exists a non-decreasing function $f:[0,\infty)\cap E \rightarrow \R$, which is continuous in $x=0$, such that
\[g(x)=\E_x\big(f(M_T)\big)\mbox{ for all }x\in [0,\infty)\cap E\] 
and write $\hat f=f-f(0)$. 
Then, the value function is given by 
\[v(x)=\E_x(\hat{f}(M_T);M_T\geq0)
\]
and for $\epsilon>0$ the impulse control strategies $S_\epsilon=(\tau_{n,\epsilon},\gamma_n)_{n=1}^\infty$, where
\[\gamma_n=0,~~\tau_{n,\epsilon}=\inf\{t>\tau_{n-1,\epsilon}:X_t\geq \epsilon\},~\tau_{0,\epsilon}=0.\]
are $\epsilon$-optimal in the sense that
\[v_\epsilon(x):=\E_x^{S_\epsilon}\left(\sum_{n=1}^\infty e^{-r\tau_{n,\epsilon}}g(X_{\tau_{n,\epsilon},-})\right)\rightarrow v(x)\;\;\mbox{ as $\epsilon\rightarrow 0$.}\]
\end{thm}

\begin{figure}
\begin{center}
\includegraphics[height=6cm]{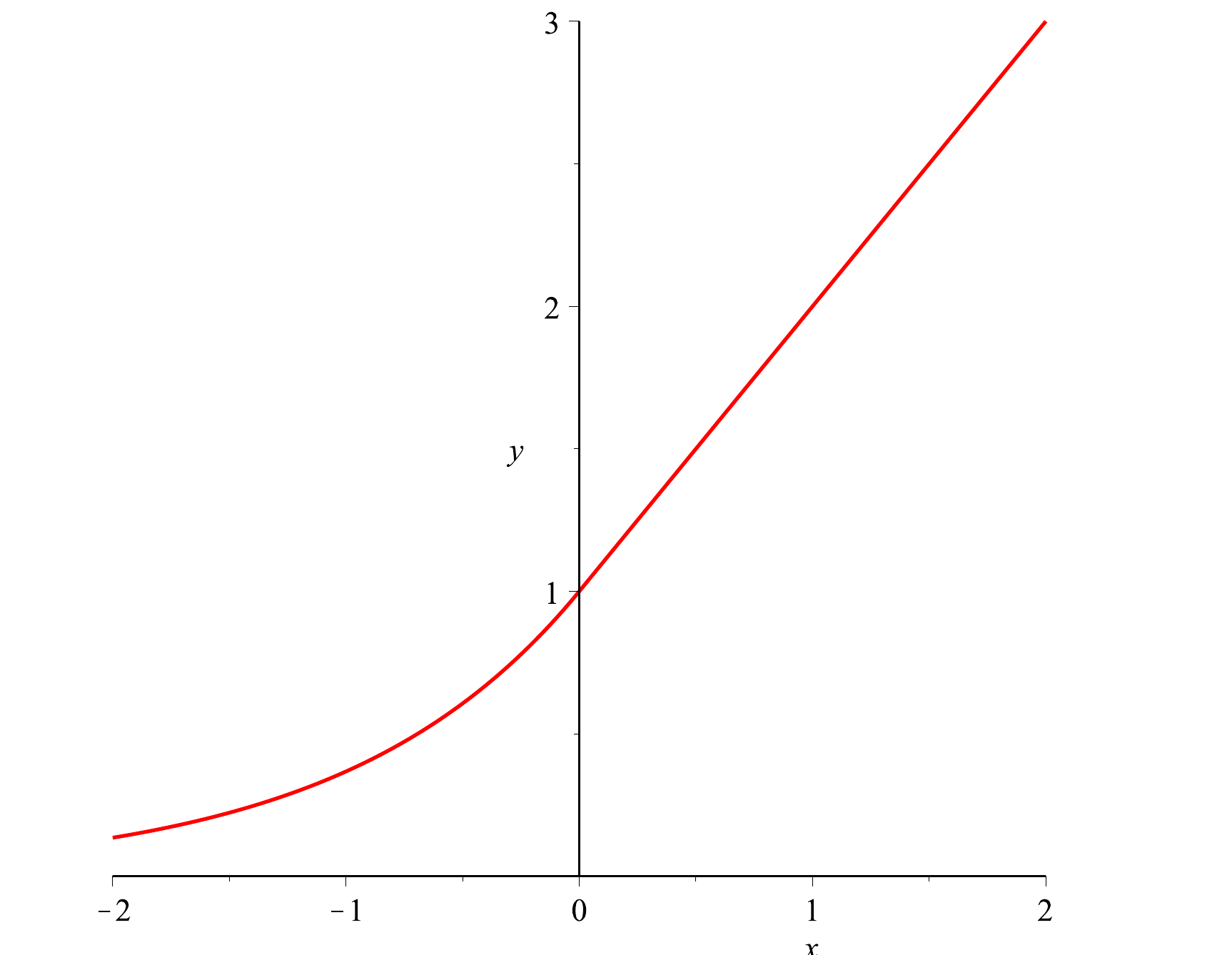}
\caption{Value function in the case $g(x)=x$ for a standard Brownian motion $X$ and $r=1/2$.}\label{fig:v1}
\end{center}
\end{figure}

\begin{proof}
Here and in the following proofs, to simplify notation, we assume that $E=\R$, so that we may write, e.g., $x>0$ to denote $x\in(0,\infty)\cap E$.
Note that, due to the right continuity of the sample paths of $X$, it holds that $\tau_{n,\epsilon}\nearrow\infty$ as $n\nearrow\infty$, so that $S_\epsilon$ is admissible in the sense given above. Writing $\hat f(x)=0$ for $x< 0,$ for short, we have to show that the function
\[\hat v(x):=\E_x\big(\hat f(M_T)\big)\]
is the value function. 
By the monotonicity of $\hat f$
\[\hat v(x)=\E_x\big(\sup_{t\leq T}\hat f(X_t)\big)\]
and by noting that $\hat f\geq 0$, we obtain from Lemma \ref{lem:sup_exc} that $\hat v$ is $r$-excessive. For $x>0$
\begin{align*}
{\bf M}\hat v(x)&=g(x)+\hat v(0)=\E_x\big(f(M_T)\big)+\E_0\big(\hat f(M_T)\big)\\
&=\E_x\big(f(M_T)\big)-f(0)+\E_0\big(f(M_T)\big)\\
&=\E_x\big(\hat f(M_T)\big)+g(0)\\
&=\hat v(x),
\end{align*}
where in the last step assumption $g(0)=0$ is used. Theorem \ref{prop:h_maj} now yields that $\hat v\geq v$. 
On the other hand, using the strong Markov property, we obtain for $x\leq 0$ and $\epsilon>0$
\begin{align*}
v_\epsilon(x)&=\E_x^{S_\epsilon}\left(\sum_{n=1}^\infty e^{-r\tau_{n,\epsilon}}g(X_{\tau_{n,\epsilon},-})\right)\\
&=\E_x^{S_\epsilon}\left(e^{-r\tau_{1,\epsilon}}g(X_{\tau_{1,\epsilon},-})\right)+\E_x^{S_\epsilon}\left(\sum_{n=2}^\infty e^{-r\tau_{n,\epsilon}}g(X_{\tau_{n,\epsilon},-})\right)\\
&=\E_x(f(M_T);M_T\geq \epsilon)+\E_x^{S_\epsilon}\left(e^{-r\tau_{1,\epsilon}}v_\epsilon(0)\right)\\
&=\E_x(f(M_T);M_T\geq \epsilon)+v_\epsilon(0)\Pro_x\left({M_T\geq \epsilon}\right),
\end{align*}
where we used Lemma \ref{lem:M_tau} in the second-to-last step. For $x=0$, we obtain
\[v_\epsilon(0)=\E_0(f(M_T);M_T\geq \epsilon)+v_\epsilon(0)\Pro_0\left({M_T\geq \epsilon}\right),\]
i.e.
\begin{align}
v_\epsilon(0)&=\frac{\E_0(f(M_T);M_T\geq \epsilon)}{1-\Pro_0\left({M_T\geq \epsilon}\right)}\label{eq:v_eps0}\\
&=\frac{\E_0(f(M_T))-\E_0(f(M_T);M_T<\epsilon)}{\Pro_0\left({M_T< \epsilon}\right)}\nonumber\\
&=\frac{g(0)-\E_0(f(M_T);M_T<\epsilon)}{\Pro_0\left({M_T< \epsilon}\right)}\nonumber \\
&=-\frac{\E_0(f(M_T);M_T<\epsilon)}{\Pro_0\left({M_T< \epsilon}\right)}.\nonumber
\end{align}
By the continuity of $f$, it holds that 
\begin{align}\label{eq:v_eps_conv}
v_\epsilon(0)\rightarrow -f(0)\mbox{ as $\epsilon\rightarrow 0$}
\end{align}
Therefore, for $x\leq 0$ as $\epsilon\rightarrow0$
\begin{align*}
v(x)\geq v_\epsilon(x)&=\E_x(f(M_T);M_T\geq \epsilon)+v_\epsilon(0)\Pro_x\left({M_T\geq \epsilon}\right)\\
&\rightarrow\E_x(f(M_T);M_T>0)-f(0)\Pro_x\left({M_T>0}\right)\\
&=\E_x(\hat{f}(M_T);{M_T>0})\\
&=\E_x(\hat{f}(M_T);{M_T\geq0})\\
&=\hat{v}(x),
\end{align*}
where we used the fact that $\hat f(0)=0$. This shows that $\hat{v}(x)\leq v(x)$ for $x\leq 0$.
{
For $x>0$ and $\epsilon\in(0,x)$, it holds that $\tau_{1,\epsilon}=0$ $\Pro_x^{S_\epsilon}$-a.s.. Therefore, letting $\epsilon\rightarrow0$ and keeping \eqref{eq:v_eps_conv} in mind,
\begin{align*}
v(x)&\geq v_\epsilon(x)=\E_x^{S_\epsilon}\left(\sum_{n=1}^\infty e^{-r\tau_{n,\epsilon}}g(X_{\tau_{n,\epsilon},-})\right)\\
&=g(x)+v_\epsilon(0)\rightarrow g(x)-f(0)=\E_x(f(M_T))-f(0)=\hat{v}(x)
\end{align*}
so that $\hat v\leq v$ everywhere. Recalling that we have already proved $\hat v\geq v$, this shows $\hat v=v$ and the $\epsilon$-optimality. }
\end{proof}

For the previous considerations, it was essential that $f$ is non-decreasing and $g(0)=0$. The cases of a non-monotonous functions $f$ and $g(0)<0$ are treated in the following subsection. But first we consider a degenerated case, where the value function is infinite. 

\begin{thm}\label{prop:degenerated_infinity}
Assume that $g(0)=0$ and there exists a continuous function $f:(0,\infty) \rightarrow \R$ such that
\[g(x)=\E_x\big(f(M_T)\big)\mbox{ for all }x> 0\] 
and
\[\lim_{x\rightarrow0}f(x)=-\infty.\]
Then, for the impulse control strategies $S_\epsilon,\;\epsilon>0,$ given by
\[\tau_{n,\epsilon}=\inf\{t>\tau_{n-1,\epsilon}:X_t\geq \epsilon\},\;\tau_{0,\epsilon}=0,\]
it holds
\[v_\epsilon(0)=\E_0^{S_\epsilon}\left(\sum_{n=1}^\infty e^{-r\tau_{n,\epsilon}}g(X_{\tau_{n,\epsilon},-})\right)\rightarrow \infty\mbox{ as $\epsilon\rightarrow 0$}.\]
Consequently, the value function is infinite for $x\geq 0$. 
\end{thm}

\begin{proof}
Arguments similar to those yielding \eqref{eq:v_eps0} in the previous proof can be used here to obtain
\begin{align*}
v_\epsilon(0)&=-\frac{\E_0(f(M_T);M_T<\epsilon)}{\Pro_0\left({M_T< \epsilon}\right)}\geq -\sup_{x\in(0,\epsilon)}f(x)\rightarrow\infty.
\end{align*}
\end{proof}

\subsection{Non-degenerated case}\label{subsec:non-deg}
In this subsection we study reward functions $g$ for which the representing function $f$ do not have to be monotonic. To be more precise, we assume the following:
\begin{assumption}\label{ass:f}
There exists a continuous function $f:[0,\infty)\cap E \rightarrow \R$ such that for all $x\geq 0$
\[g(x)=\E_x\big(f(M_T)\big)\]
with\ $g(0)\leq 0$. Moreover, there exists a global minimum point $\un{x}\in[0,\infty)\cap E$ of $f$ such that
\begin{itemize}
\item $f(\un x)<0$,
\item $f$ is strictly increasing on $[\un x,\infty)\cap E$,
\item $f(x)>0$ for $x$ sufficiently large.
\item if $g(0)=0$, then\ $\un x>0$.
\end{itemize}
\end{assumption}


\begin{figure}
\begin{center}
\begin{tikzpicture}
\begin{axis}[
]
\addplot[red,domain=0:2.5,samples=201]
{((x -1)^2-1)};
\addplot[black,domain=0:2.5,samples=201]
{(-.6)};
\addplot+[nodes near coords,only marks,
point meta=explicit symbolic]
table[meta=label] {
x y label
1.63 -0.6 $(x_c,-c)$
2 0 $(\ov x,0)$
1 -1 $(\un x,-c^*)$
0.37 -0.6 $(\un x_{c},-c)$
};
\end{axis}
\end{tikzpicture}
\caption{Plot of an Appell polynomial $f=Q_m$ for $m=2$ (red).}\label{fig:v2}
\end{center}
\end{figure}
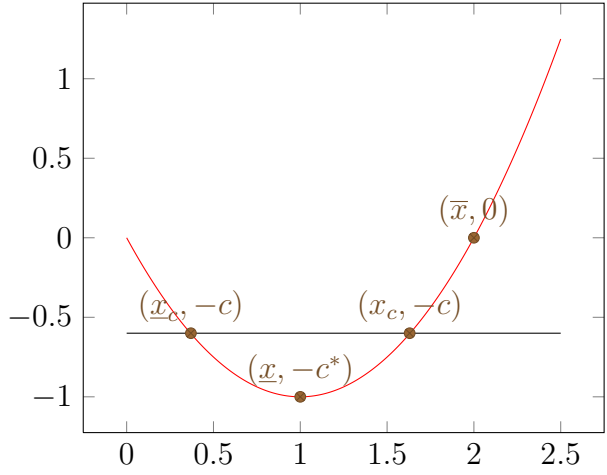

Let $c^*:=-f(\un x)$. By the intermediate value theorem and the monotonicity, for each $c\in[0,c^*]$ there is a unique $x_c\in[\un x,\ov x]$ such that $f(x_c)=-c$, where $\ov x$ denotes the unique root of $f$ in $[\un x,\infty)$, which exists by the intermediate value theorem. Now, consider the following candidates for the value function:
\[\hat{v}_c(x)=\E_x\left(\hat f_c(M_T);{M_T\geq x_c}\right),\]
where $\hat f_c(x)=f(x)+c$ and $c\in[0,c^*]$.

%

\begin{lemma}\label{lem:v_c}
Under Assumption \ref{ass:f} and the notation given above, for each\ $c\in[0,c^*]$ it holds that
\begin{enumerate}[(i)]
  \item $\hat{v}_c(x)=\E_x\big(\sup_{t\leq T}(\hat f_c(X_t)\mathds{1}_{\{X_t\geq x_c\}})\big)$ for all $x$,
  \item\label{item:exc} $\hat{v}_c$ is non-negative and $r$-excessive for the underlying uncontrolled process,
  \item\label{item:v_c3} $\hat{v}_c$ has the $r$-harmonicity property on $(-\infty,x_c)$, i.e. 
  \[\hat{v}_c(x)=\E_x\big(e^{-r\tau}\hat{v}_c(X_\tau)\big),\]
  where $\tau=\inf\{t\geq0:X_t\geq x_c\}$ and $x\leq x_c$,
  \item\label{item:Mv} $\hat{v}_c(x)=c+g(x)$ for $x\geq x_c$,
  \item there exists $\hat c\in(0,c^*]$ such that $\hat c=\hat{v}_{\hat c}(0)$.
\end{enumerate}
\end{lemma}
\begin{proof}
$(i)$ follows from the monotonicity assumption on $f$, $(ii)$ is obtained from Lemma \ref{lem:sup_exc} are clear recalling the monotonicity assumptions on $f$ and the results of Subsection \ref{subsec:hunt}. $(iii)$ can easily be obtained from \eqref{eq:harm_threshold}. For $(iv)$ notice that for $x\geq x_c$
\begin{align*}
\hat{v}_c(x)&=\E_x(\hat{f}_c(M_T);M_T\geq x_c)=\E_x({f}(M_T);M_T\geq x_c)+c\Pro_x(M_T\geq x_c)\\
&=g(x)+c.
\end{align*}
For $(v)$ consider the function\ $z:c\mapsto c-\hat{v}_c(0)$. Since $f\geq 0$ on $[\ov x,\infty)$ it holds that
\[z(0)=-\hat v_0(0)=-\E_0\left(f(M_T);{M_T\geq \ov x}\right)\leq 0\]
and, on the other hand, since $f(x)+c^*\geq 0$ for all $x\geq0$
\begin{align*}
z(c^*)&=c^*-\E_0\left(f(M_T)+c^*;{M_T\geq \un x}\right)\geq c^*-\E_0\left(f(M_T)+c^*\right)\\
&=-\E_0(f(M_T))=-g(0)\geq 0,
\end{align*}
where, due to Assumption \ref{ass:f}, the first inequalityy is strict if $\underline{x}>0$ and the second if $\underline{x}=0$. 
Now, the intermediate value theorem yields the existence of $\hat c$. 
\end{proof}

\begin{remark}\label{rem:x_c}
Recall that
\[{\bf M}\hat{v}_c(x_c)=g(x_c)+\hat{v}_c(0)=\E_{x_c}(f(M_T))+\E_0[(f+c)(M_T);M_T\geq x_c]\]
and 
\[\hat{v}_c(x_c)=\E_{x_c}[(f+c)(M_T);M_T\geq x_c]=\E_{x_c}[(f+c)(M_T)]=g(x_c)+c.\]
Therefore, the value $\hat c$ in Lemma \ref{lem:v_c} $(v)$ can be characterized also via
\begin{align*}
{\bf M}\hat{v}_{\hat c}(x_{\hat c})=\hat{v}_{\hat c}(x_{\hat c}).
\end{align*}
\end{remark}

\begin{notation}\label{not:def_vhat}
Our candidate for the value function is now given by
\begin{align*}
\hat v:=\hat{v}_{\hat{c}},
\end{align*} 
where $\hat{c}$ is chosen to be the smallest one satisfying the condition 
given in Lemma \ref{lem:v_c} $(v)$, and we also write 
\[\hat{f}:=\hat{f}_{\hat{c}}=f+\hat{c}\mbox{ and }x^*=x_{\hat{c}}>0,\]
for short.
\end{notation}

We need one further assumption:
\begin{assumption}\label{eq:Mv_v}
For all $x\in(0,\un x)$
\begin{align}
\E_x(\hat f (M_T);M_T\leq x^*)\leq 0.\label{eq:Mv_v2}
\end{align}
\end{assumption}

\begin{remark}\label{rem:Mv_v}
Notice that for all\ $x\leq x^*$ we have
\begin{align*}
\E_x(\hat f (M_T);M_T\leq x^*)&=\E_x(\hat f (M_T))-\E_x(\hat f (M_T);M_T\geq x^*)\\
&=\E_x(f(M_T))+\hat c-\E_x(\hat f (M_T);M_T\geq x^*)\\
&=g(x)+\hat c-\E_x(\hat f (M_T);M_T\geq x^*),
\end{align*}
and, in particular, for $x=0$ using Lemma \ref{lem:v_c} $(iii)$
\begin{align}\label{eq:rem:Mv_v}
\E_0(\hat f (M_T);M_T\leq x^*)=g(0)+\hat c-\E_0(\hat f (M_T);M_T\geq x^*)=g(0)\leq 0.
\end{align}
Furthermore, \eqref{eq:Mv_v2} also trivially holds true for all $x\geq \un x$ due to the assumption $\hat{f}=f+\hat{c}<0$ on $[\un{x},x^*]$.
\end{remark}


\begin{thm}\label{thm:nondegenerated}
Let Assumptions \ref{ass:f} and \ref{eq:Mv_v} hold true.
Define a sequence $S^*=(\tau_n^*,\gamma_n^*)$ by $\tau_0^*:=0,$ 
\[\tau_n^*:=\inf\{t\geq \tau_{n-1}^*:X_t\geq x^*\},\;n=1,2,\dots\]
and $\gamma_n^*=0$. Then the sequence $S^*$ is an optimal impulse control sequence for the problem \eqref{eq:problem} and the value function is given by
\[v(x)=\hat v(x)=\E_x(\hat f(M_T);\ M_T\geq x^*)\]
with $\hat{c}$ as in Notation \ref{not:def_vhat}.
\end{thm}

\begin{proof}
As \eqref{eq:Mv_v2} clearly holds true also for all $x\geq \un x$, see Remark \ref{rem:Mv_v}, we have
$\mbox{for all }x>0$
\begin{align*}
\hat{v}(x)-{\bf M}\hat{v}(x)&=\hat{v}(x)-(g(x)+\hat{v}(0))\\
&=\E_x\left(\hat f(M_T);{M_T\geq x^*}\right)-\E_x\left(f(M_T)\right)-\hat{c}\\
&=\E_x\left(\hat f(M_T);{M_T\geq x^*}\right)-\E_x\left(\hat f(M_T)\right)\\
&=-\E_x\left(\hat f(M_T);{M_T\leq x^*}\right)\\
&\geq 0
\end{align*}
with equality for $x\geq x^*$, which is also known from Lemma \ref{lem:v_c} (iv). For $x\leq 0$, the inequality holds as ${\bf M}\hat{v}(x)=-\infty$ for $x<0$. Keeping Lemma \ref{lem:v_c}\eqref{item:exc} in mind, the assumptions of Theorem \ref{prop:h_maj} (i) are fulfilled, which yields that $\hat v\leq v$.

For the reverse inequality, we check that the assumption of Theorem \ref{prop:h_maj}\eqref{item:h_maj2} are fulfilled. First note that the sequence $S^*$ is obviously admissible. By the strong Markov property and the definition of $S^*$, we have to prove that
\begin{align}
&\hat{v}(0)=\E_0(e^{-r\tau_1^*}\hat v(X_{\tau_1^*}))\label{eq:cond1}\\
&\hat{v}(0)+g(y)=\hat{v}(y)\mbox{ for all }y\geq x^*.\label{eq:cond2}
\end{align}
Indeed, \eqref{eq:cond1} and \eqref{eq:cond2} hold by Lemma \ref{lem:v_c} \eqref{item:v_c3} and \eqref{item:Mv}, respectively.  
\end{proof}

\subsection{On Assumption \ref{eq:Mv_v}}\label{subsec:cond_6}
To apply Theorem \ref{thm:nondegenerated} one has to make sure that Assumption \ref{eq:Mv_v} holds true. In this subsection, we discuss conditions under which this is the case. The first such condition is trivial, but nonetheless useful in many situations where $g(0)<0$. 
\begin{proposition}\label{prop:un_x_0}
If $\hat f \leq 0$ on $(0,\un x)$ and Assumption \ref{ass:f} is valid, then also Assumption \ref{eq:Mv_v} holds true. In particular, this is the case if $\un x=0$, i.e. $f$ is strictly increasing and $g(0)<0$. 
\end{proposition}
\begin{proof}
When $\hat f \leq 0$ on $(0,\un x)$, it is also nonnegative on $[0,x^*]$ by Assumption \ref{ass:f}. Therefore,
\begin{align*}
\E_x(\hat f (M_T);M_T\leq x^*)\leq 0\mbox{ for all }x\in(0,\un x).
\end{align*}
\end{proof}

Unfortunately, in some situations of interest, the function\ $\hat{f}=f+\hat c$ is \emph{not} nonnegative on $[0,x^*]$. In particular, this is the typical situation in the case $g(0)=0$. (Note that for $g(0)=0$, Assumption \ref{eq:Mv_v} holds with equality for $x=0$, see \eqref{eq:rem:Mv_v}, which makes it impossible that $\hat f\leq 0$ in all nontrivial cases.) In the following, we find sufficient conditions to guarantee the validity of Assumption \ref{eq:Mv_v} also in these cases when the underlying process is a L\'evy process or a diffusion.

\begin{proposition}\label{prop:density_decreasing}
Let\ $X$ be a L\'evy process and Assumption \ref{ass:f} be valid. 
It is, furthermore, assumed that $\hat{f}$ is decreasing on $[0,\un x]$ and that the distribution of $M_T$ under $\Pro_0$ is given by a possible atom at 0 and a Lebesgue density $b$ on $(0,\infty)$, 
which is \emph{non-increasing}.
Then Assumption \ref{eq:Mv_v} holds true.
\end{proposition}

\begin{proof}
Let $x\in (0,\un x_{\hat c})$, where $\un x_{\hat c}$ denotes the smallest root of $\hat f$ in $(0,\un x)$ and write
\begin{align*}
z(x)&:=\E_x(\hat f(M_T);\ M_T\leq x^*).
\end{align*}
We have to prove that $z(x)\leq 0$. From \eqref{eq:rem:Mv_v}, we already know that $z(0)\leq 0$. Writing $h:=\hat f1_{\{x\leq x^*\}}$ and keeping in mind that $h$ is nonnegative and decreasing on $(0,\un x_{\hat c})$ and nonpositive on $(\un x_{\hat c},\infty)$, it holds that
\begin{align*}
0&\geq z(0)=\E_0 h(M_T)=h(0)\Pro_0(M_T=0)+\int_0^\infty h(y)b(y)dy\\
&=h(0)\Pro_0(M_T=0)+\int_0^{\un x_{\hat c}} h(y)b(y)dy+\int_{\un x_{\hat c}}^\infty h(y)b(y)dy\\
&\geq h(x)\Pro_0(M_T=0)+\int_0^{\un x_{\hat c}-x} h(y)b(y)dy+\int_{\un x_{\hat c}}^\infty h(y)b(y-x)dy\\
&\geq  h(x)\Pro_0(M_T=0)+\int_0^{\un x_{\hat c}-x} h(x+y)b(y)dy+\int_{\un x_{\hat c}-x}^\infty h(x+z)b(z)dz\\
&= h(x)\Pro_0(M_T=0)+\int_0^{\infty} h(x+y)b(y)dy=\E_0 \big(h(x+M_T)\big)=z(x).
\end{align*}
\end{proof}

\begin{remark}
The assumption that $M_T$ has a Lebesgue density $b$ which is {non-increasing} with a possible atom at 0 holds true, e.g., for all spectrally negative L\'evy processes where $M_T$ is exponentially distributed. It furthermore holds for processes with mixed exponential upward jumps, see \cite{Mordecki02}.
\end{remark}

Now, we state a similar result, that typically holds true for Markov processes with no positive jumps.

\begin{proposition}\label{prop:suff_cond_diff}
~\\[-.6cm]
\begin{enumerate}[(i)]
\item Assume that 
\begin{equation}\label{eq:decomp_M_T}
\Pro_x(M_T\in dy)=\psi_r(x)\sigma(dy),\;y\geq x,\end{equation}
where\ $\psi_r$ denotes an increasing function and $\sigma$ is a measure that does not depend on $x$. Under Assumption \ref{ass:f} 
it is, furthermore, assumed that $\hat{f}$ changes sign only once in $(0,\un x)$ (from + to -). 
Then Assumption \ref{eq:Mv_v} holds true.
\item A decomposition of the type \eqref{eq:decomp_M_T} holds true for all regular diffusions and all spectrally negative L\'evy processes.
\end{enumerate}
\end{proposition}

\begin{proof}
To prove $(i)$, let $x\in (0,\un x_{\hat c})$, where $\un x_{\hat c}$ denotes the smallest root of $\hat f$ in $(0,\un x)$. By assumption $\hat{f}(y)\geq 0$ on $(0,\un x_{\hat c})$ and $\hat{f}(y)\leq 0$ on $[\un x_{\hat c},x^*]$ and write
\begin{align*}
z(x)&:=\E_x(\hat f(M_T);\ M_T\leq x^*)
\end{align*}
and we obtain that 
\[z(x)=\psi_r(x)\int_x^{ x^*}\hat{f}(y)\sigma(dy).\]
Obviously, $z(x)\leq 0$ on $[\un x_{\hat c},x^*]$ and, by \eqref{eq:rem:Mv_v}, $z(0)\leq0$. Therefore, we shall prove that
\[w(x):=\int_x^{ x^*}\hat{f}(y)\sigma(dy)\]
is non-positive on $[0,\un x_{\hat{c}}]$, but since $\hat{f}$ is nonnegative on this interval, $w$ is decreasing there. By Remark \ref{rem:Mv_v} we know that $w(0)\leq0$. This proves the assertion. 

 For (ii) recall that the distribution of $M_T$ for a regular diffusion processes is given by (see \cite{BS}, p. 26):
\[\Pro_x(M_T\in dy)=\psi_r(x)\sigma(dy),\;y\geq x,\]
where\ $\psi_r$ denotes the increasing fundamental solution for the generalized differential equation associated with $X$ and the measure $\sigma$ does not depend on $x$ and is given by
\[\sigma([y,\infty))=\frac{1}{r\psi_r(y)}.\] 
For spectrally negative L\'evy processes, it is well-known that $M_T$ is Exp($\theta$)-distributed, where $\theta$ denotes the unique root of the equation $\Psi(\lambda)=r$, where $\Psi$ denotes the L\'evy exponent of $X$, see \cite{kyp}, p. 213. Therefore,
\[\Pro_x(M_T\geq y)=\Pro_0(M_T\geq y-x)=e^{-\theta(y-x)}=e^{\theta x}e^{-\theta y},\]
which yields the desired decomposition.
\end{proof}

\begin{remark}
Indeed, the assumption \eqref{eq:decomp_M_T} holds for wide classes of Hunt processes with no positive jumps under minimal assumptions, see, e.g., the discussion in \cite{Cisse2012}, Proposition 4.1. 
\end{remark}

For underlying diffusion processes, the advantage in the treatment of impulse control problems is that no overshoot occurs. This allows for using techniques which are very different in nature to the approach we use here. To see that the results obtained using special diffusion techniques are essentially covered by our results, we use the following characterization of our Assumption \ref{eq:Mv_v}. Equation \eqref{eq:5diff} below, which is a consequence of the quasi-variational inequalities, plays an essential role in \cite{Alvarez04}. 

\begin{proposition}
Let\ $X$ be a regular diffusion process and let Assumption \ref{ass:f} hold true. Then Assumption \ref{eq:Mv_v} holds true if and only if
\begin{align}\label{eq:5diff}
\psi_r(x)\frac{g(x^*)-g(0)}{\psi_r(x^*)-\psi_r(0)}\geq \frac{\psi_r(0)g(x^*)-\psi_r(x^*)g(0)}{\psi_r(x^*)-\psi_r(0)}+g(x)\mbox{ for all }x\in[0,x^*],
\end{align}
where\ $\psi_r$ denotes the increasing fundamental solution for the generalized differential equation associated with $X$.
\end{proposition}

\begin{proof}
Recalling Remark \ref{rem:Mv_v}, we see that Assumption \ref{eq:Mv_v} is equivalent to
\begin{align}\label{eq:5diff2}
\E_x\left(\hat f(M_T);{M_T\geq x^*}\right)-g(x)-\hat c\geq 0.
\end{align}
As above, we use the representation of the distribution of $M_T$ from \cite{BS}, p. 26:
\[\Pro_x(M_T\in dz)=\psi_r(x)\sigma(dz).\]
Writing
\[\Delta=\int_{x^*}^\infty \hat f(z)\sigma(dz),\]
the inequality \eqref{eq:5diff2} for $x\leq x^*$ reads as $\psi_r(x)\Delta\geq g(x)+\hat c.$ Recalling that $\hat c$ and $x^*$ were chosen so that we have equality for $x=x^*$ and $x=0$, a short calculation yields that $c=-g(0)+\psi_r(0)\Delta$ and 
\[\Delta=\frac{g(x^*)-g(0)}{\psi_r(x^*)-\psi_r(0)}.\]
Inserting this and rearranging terms, we obtain that $\psi_r(x)\Delta\geq g(x)+\hat c$ holds if and only if
\begin{align*}
\psi_r(x)\frac{g(x^*)-g(0)}{\psi_r(x^*)-\psi_r(0)}\geq \frac{\psi_r(0)g(x^*)-\psi_r(x^*)g(0)}{\psi_r(x^*)-\psi_r(0)}+g(x).
\end{align*}\end{proof}
%
We remark that analytical conditions on $g$ and $X$ can be found to guarantee that \eqref{eq:5diff} holds true. This was carried out in \cite{Alvarez04}, Lemma 5.2, where three such conditions are given. (Note that in the setting discussed there, it is assumed that $g(0)< 0$, while we consider also the case $g(0)= 0$.)


\section{Examples}\label{sec:example}
\subsection{Power reward for geometric L\'evy processes}\label{subsec:geo_levy_power}
We consider an extension of the problem treated in \cite{Alvarez04}, Section 6.1, where the case of a geometric Brownian motion and a power reward function was studied. More generally, we consider a general L\'evy process $X$, which we assume not to be a subordinator, and reward function $g(x)=e^{bx}-k$ for some $b>0,\;k>1$. We assume the integrability condition $\E_0(e^{bX_1})<{e}^r$ to hold true. Using the systematic approach from \cite{CST}, or just by guessing, we see that the function
\[f(x)=ae^{bx}-k,\]
where $a={1}/{\E_0e^{bM_T}}$, fulfills Assumption \ref{ass:f} with\ $\un x=0$. Recalling Proposition \ref{prop:un_x_0}, we see that the assumptions of Theorem \ref{thm:nondegenerated} are fulfilled. For $c>0$, the equation $f(x)=-c$ has the unique solution 
\begin{align}x_c=\frac{\log((k-c)/a)}{b}\label{eq:x_c_geometric}
\end{align}
and the optimal value $\hat c$ is given by the equation
\begin{align}
\hat c&\stackrel{}{=} \hat{v}_{\hat c}(0)=\E_0\left(\hat f_{\hat c}(M_T);{M_T\geq x_{\hat c}}\right)\label{eq:hat_c_mixed}\\
&=\int_{x_{\hat c}}^\infty\big(ae^{by}+{\hat c}\big)\Pro_0(M_T\in dy).\nonumber
\end{align}
How explicit this equation can be solved now depends on the distribution of $M_T$. For example, in the case that $X$ has arbitrary downward jumps, a Brownian component, and mixed exponential upward jumps, it is known that $M_T$ has a Lebesgue density of the form 
\[\sum_{k=1}^{n+1} \zeta_k\eta_ke^{-\eta_ky}\mbox{ for }y>0,\]
where $\eta_{1},...,\eta_{n+1}$  and $\zeta_1,...,\zeta_{n+1}$ are positive constants (see \cite{Mordecki02}). In this case, \eqref{eq:hat_c_mixed} reads as
\[\hat{c}=\sum_{k=1}^{n+1}\zeta_k\left(a\frac{\eta_k}{b-\eta_k}e^{(b-\eta_k)x_{\hat c}}+\hat c e^{-\eta_kx_{\hat c}}\right),\]
which may be solved numerically for $\hat c$. {This also yields a value for the optimal boundary $x^*$.}
Summarizing the results, we obtain
 \begin{proposition}
Define $S^*=(\tau_n^*,\gamma_n^*)_{n=1}^\infty$ by setting for all $n=1,2,...$
\[\gamma_n^*=0,~~\tau_n^*:=\inf\{t\geq \tau_{n-1}^*:X_t\geq x^*\},~~\tau_0=0.\] Then the sequence $S^*$ is an optimal impulse control sequence for the power reward impulse control problem for geometric L\'evy processes. The value function is
\[v(x)=\E_x\big(ae^{bM_T}-k+\hat{c};\ M_T\geq x^*\big)\]
with $\hat{c}$ given by \eqref{eq:hat_c_mixed} and $x^*=x_{\hat c}$ given by \eqref{eq:x_c_geometric}.
\end{proposition}

\subsection{Power reward for L\'evy processes}\label{subsec:levy_power}
Now, we consider the case $g(x)=x^m,\;m\in\N,\;m\geq 2$ for a L\'evy process $X$ which we assume not to be a subordinator, and the L\'evy measure $\pi$ satisfies 
\begin{equation}
\label{finite_exp}
\int_{(-\infty, -1)\cup(1, +\infty)}|y|^m\pi(dy)<\infty.
\end{equation}
This assumption implies that $X_T$ has the finite $m$-th moment.

We start with the case $m=1$. Then, it is clear that
\[g(x)=\E_x\big(Q_1(M_T)\big),\]
where $Q_1(x)=f(x)=x-\E_0(M_T)$ is the first Appell polynomial associated with $M_T$. For this function, the assumptions of Theorem \ref{thm:degenerated} are obviously fulfilled yielding the following:
\begin{proposition}
For a L\'evy process $X$ as above, the impulse control problem with reward function $g(x)=x$ has value function
\[v(x)=\E_x(\hat{f}(M_T);{M_T\geq0})=\begin{cases}
x+\E_0(M_T),&\;x\geq 0,\\
\E_x(M_T^+),&\;x<0,
\end{cases}\]
and the impulse control strategies $S_\epsilon$ given by
\[\tau_{n,\epsilon}=\inf\{t>\tau_{n-1,\epsilon}:X_t\geq \epsilon\},\;\tau_{0,\epsilon}=0,\]
are $\epsilon$-optimal.
\end{proposition}

As the optimal strategy is degenerated for  L\'evy process $X$ and reward $g(x)=x$, this is not the case for other powers $g(x)=x^m,\;m\in\N,\;m\geq 2$. In the following lemma, we collect two well-known properties of the $m$-th Appell polynomial $Q_m$ associated with $M_T$, see \cite{NS,NS2} or  \cite{Sa}.

\begin{lemma}\label{lem:appell-poly}
For $m\in\N,\;m\geq 2$, the following holds true:
\begin{enumerate}[(i)]
  \item $Q_m$ fulfills Assumption \ref{ass:f}.
  \item $Q_m$ is decreasing on $[0,\un x]$
\end{enumerate}
\end{lemma}

 Using this lemma, we may apply Theorem \ref{thm:nondegenerated}  to obtain the following solution for the power reward problem for general L\'evy processes:
 
 \begin{proposition}
Let $m\in\N,\;m\geq2,$ and let Assumption \ref{eq:Mv_v} hold true.
Define a sequence $S^*=(\tau_n^*,\gamma_n^*)$ by $\tau_0^*:=0,$ 
\[\tau_n^*:=\inf\{t\geq \tau_{n-1}^*:X_t\geq x^*\}\mbox{ for all }n\in\N,\]
and $\gamma_n^*=0$. Then the sequence $S^*$ is an optimal impulse control sequence for the power reward impulse control problem. The value function is given by
\[v(x)=\E_x(Q_m(M_T)+\hat{c};\ M_T\geq x^*)\]
with $\hat{c}$ as in Notation \ref{not:def_vhat}.
\end{proposition}
\begin{remark}
Whenever the distribution of $M_T$ is given by a possible atom at 0 and a Lebesgue density $b$ on $(0,\infty)$, 
which is {non-increasing}, then the conclusion of the previous proposition holds true by Proposition \ref{prop:density_decreasing}.
\end{remark}

%

To obtain more explicit results, we now concentrate on the case of spectrally negative L\'evy processes{, including, e.g., all Brownian motions with drift.
For spectrally negative L\'evy processes}, it is well-known that $M_T$ is Exp($\theta$)-distributed, where $\theta$ denotes the unique root of the equation $\Psi(\lambda)=r$ and $\Psi$ denotes the L\'evy exponent of $X$, see \cite{kyp}, p. 213. 
From \cite{Sa}, Section 2.2, we know that the $m$-th Appell polynomial $f=Q_m$ is given by 
\[f(x)=\left(x-\frac{m}{\theta}\right)x^{m-1},\]
and $x_c$ is therefore given as the unique positive root of
\[\left(x-\frac{m}{\theta}\right)x^{m-1}=-c\]
for all $c\in[0,c^*]$. In particular, in the case\ $m=2$, we have
\[x_c=(1+\sqrt{1-\theta^2 c})/\theta.\]
Recalling that $M_T$ is Exp($\theta$)-distributed and the process does not jump upwards, a short calculation yields that the condition ${\hat{c}}=\hat{v}_{\hat{c}}(0)$ becomes
\begin{align}
{\hat{c}}&=\E_0\left((f+\hat{c})(M_T);M_T\geq x_{\hat{c}}\right)=(g(x_{\hat{c}})+\hat{c})e^{-\theta x_{\hat{c}}}\label{eq:appel_m}\\
&=(x_{\hat{c}}^m+\hat{c})e^{-\theta x_{\hat{c}}}\nonumber
\end{align}
and the value function is given by
\begin{align}\label{eq:value_power_spectr}
v(x)=\begin{cases}
({x^*}^m+\hat{c})e^{-\theta (x^*-x)},&x<x^*,\\
\hat c+x^m,&x\geq x^*.
\end{cases}
\end{align}
For $m=2$, the unique solution $c$ of Equation \eqref{eq:appel_m} is $\hat{c}=\theta^{-2}w$, where $w$
is the unique solution to the equation
$$
2(1+\sqrt{1-y})=y\,\e^{1+\sqrt{1-y}},\quad 0<y<1.
$$
This equation can be solved explicitly in terms of the Lambert W-function, see \cite{CGHJK}:
\[w=-\LambertW(-2\e^{-2})^2-2\LambertW(-2\e^{-2}).\]
Now, the optimal strategy is given by the threshold
\[x^*=(1+\sqrt{1-w})/\theta.\]
{
It is easily checked that the value function $v$ in \eqref{eq:value_power_spectr} has the smooth fit property in\ $x^*$, see also Figure \ref{fig:v3}. As the restarting state $0$ is fixed, we of course cannot expect to have a smooth fit there, but the value function intersects the reward by construction.
}

\begin{figure}
\begin{center}
\includegraphics[height=6cm]{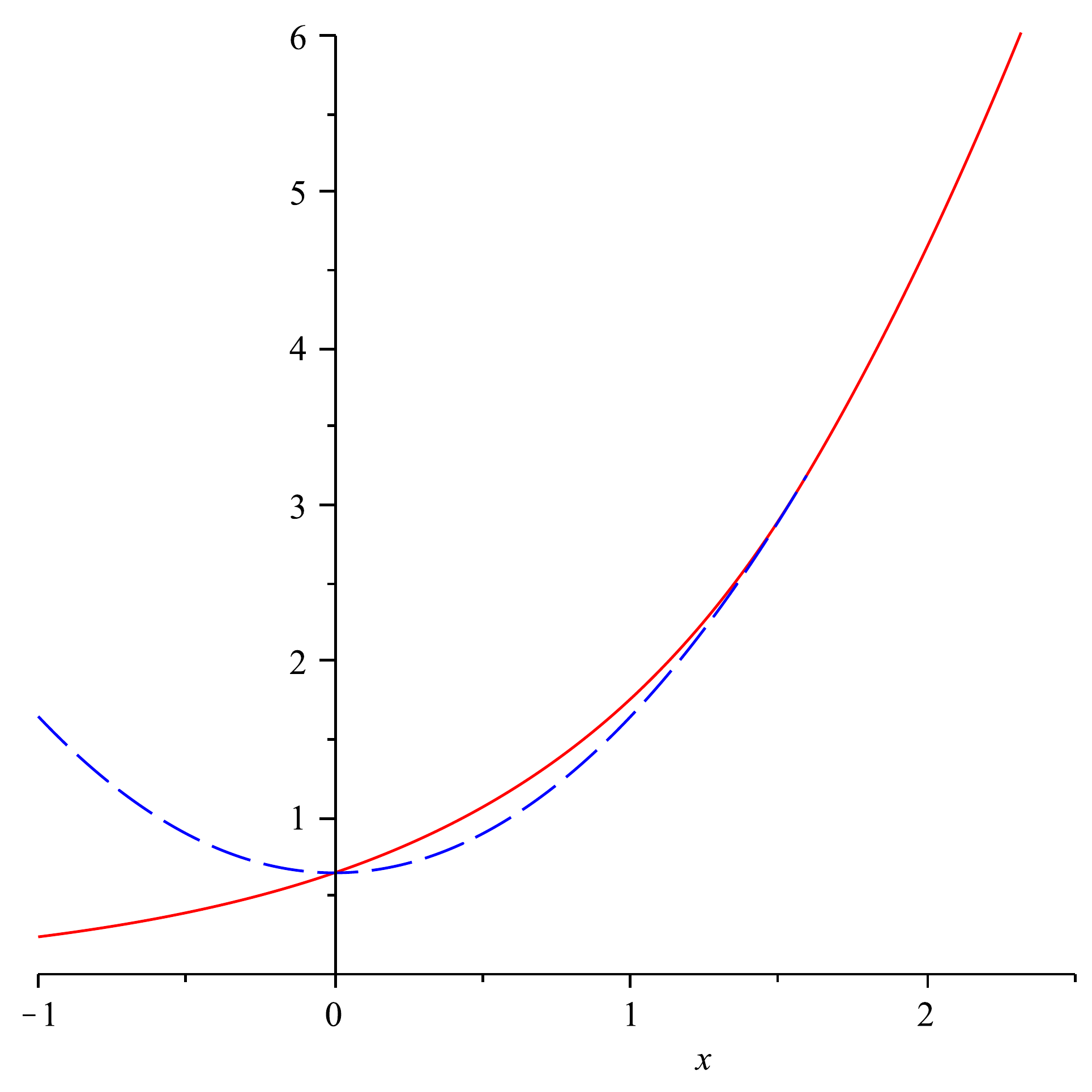}
\caption{Value function in the case $g(x)=x^2$ for a standard Brownian motion $X$ and $r=1/2$. Here,\ $x^*\approx 1.59$.}\label{fig:v3}
\end{center}
\end{figure}

\subsection{Power reward for reflected Brownian motion}
We consider the power reward function $g(x)=x^m,\;m\in\N,$ for a Brownian motion reflected in 0. For the case $m\geq 2$, using the method described in \cite{CST}, Subsection 2.2, and the explicit expressions from \cite{BS}, A1.2, it is straightforward to find the representing function $f=f_m$ as
\[f_m(x)=x^{m-1}\big(x-\frac{m}{\sqrt{2r}}\coth(x\sqrt{2r})\big),\]
which indeed fulfills 
\[x^m=\E_x\big(f_m(M_T)\big)\mbox{ for all }x\geq 0.\] 
The case $m=1$ has to handled with more care, which is due to a local time term arising in this case. Nonetheless, it turns out that the same representation holds also in the case $m=1$ for $x>0$, i.e.
\[x=\E_x\big(f_1(M_T)\big)\mbox{ for all }x> 0.\] 
For a detailed treatment, see \cite{Bao_thesis}.

The impulse control problem can be solved by curve sketching of $f_m$. We start with the case $m=1$. Here, $f_1$ is increasing and $\lim_{x\rightarrow0}f_1(x)=-\infty$. Therefore, Theorem \ref{prop:degenerated_infinity} is applicable and we see that the value is infinite. \\
For $m=2$, $f_2$ is also increasing, with $f_2(0):=\lim_{x\rightarrow 0}f_2(x)=-1/r$. Therefore, we are in the degenerated case with finite value. Theorem \ref{thm:degenerated} yields $\epsilon$-optimal strategies and
\[v(x)=x^2+v(0)=x^2+1/r\mbox{ for all }x\geq0.\]
For $m\geq 3$, we obtain an optimal non-degenerated impulse-control strategy. First, $x_c$ is the unique positive root of
\[x^{m-1}\big(x-\frac{m}{\sqrt{2r}}\coth(x\sqrt{2r})\big)=-c\]
for $c\in[0,c^*]$. From \cite{BS}, p. 411, the distribution of $M_T$ is given by
\[\Pro_x(M_T\in dy)=\sqrt{2r}\frac{\sinh(y\sqrt{2r})\cosh(x\sqrt{2r})}{\cosh(y\sqrt{2r})^2}dy.\]
We calculate
\begin{align*}
&\E_0\left((f+c)(M_T);M_t\geq x_c\right)\\
&=\int_{x_c}^\infty \big(y^{m-1}\big(y-\frac{m}{\sqrt{2r}}\coth(y\sqrt{2r})\big)+c)\sqrt{2r}\frac{\sinh(y\sqrt{2r})}{\cosh(y\sqrt{2r})^2}dy\\
&=-\left[\frac{2\e^{y\sqrt{2r}}(c+y^m)}{\e^{2y\sqrt{2r}}+1}\right]_{x_c}^\infty
=\frac{2\e^{x_c\sqrt{2r}}(c+x_c^m)}{\e^{2x_c\sqrt{2r}}+1}.
\end{align*}
Therefore, the condition $c=\hat{v}_c(0)$ reads as
\[c=\frac{2\e^{x_c\sqrt{2r}}(c+x_c^m)}{\e^{2x_c\sqrt{2r}}+1},\]
which may be solved numerically to obtain $\hat{c}$ and $x_{\hat{c}}=x^*$. No matter what the exact value of $\hat{c}$ is, the assumption of Proposition \ref{prop:suff_cond_diff} are fulfilled, so that Theorem \ref{thm:nondegenerated} can be applied, which yields the optimal impulse control strategy with $x^*$ as a threshold.
\bibliographystyle{abbrvnat}
\bibliography{literatur}

\end{document}